\documentclass[reqno]{amsart}
\usepackage{amsmath,amsthm,amsfonts,amssymb}
\RequirePackage{graphicx}
\numberwithin{equation}{section}
\theoremstyle{plain}

\newtheorem{theorem}{Theorem}

\newtheorem{definition}[theorem]{Definition}

\newtheorem{lemma}[theorem]{Lemma}

\newtheorem{idealization}[theorem]{Idealization}
\newtheorem{algorithm}[theorem]{Algorithm}

\newtheorem{assumption}[theorem]{Assumption}

\begin{document}

\title[ParRep for Markov chains]{The parallel replica method for computing equilibrium averages of Markov chains}

\author{DAVID ARISTOFF}

\address
{Department of Mathematics\newline
\indent Colorado State University}
\email{
aristoff@math.colostate.edu
}

\date{31 December 2014.}


\begin{abstract} An algorithm is proposed for computing  
equilibrium averages of Markov chains which suffer from metastability -- 
the tendency to remain in one or more subsets of state space for 
long time intervals. The algorithm, called the parallel replica method (or 
ParRep), uses many parallel 
processors to explore these subsets more efficiently. 
Numerical simulations on a simple model demonstrate consistency of 
the method. A proof of consistency is given in an idealized setting. 
The parallel replica method can be considered a generalization of A.F. Voter's 
parallel replica dynamics, originally developed to efficiently 
simulate metastable Langevin stochastic dynamics. 
\end{abstract}

\maketitle

\section{Introduction}

This article concerns the problem of computing equilibrium averages 
of time homogeneous, ergodic Markov chains in the presence of 
metastability. A Markov chain is said to be {\em metastable} if it 
has typically very long sojourn times in certain subsets of 
state space,  called {\em metastable sets}. 
A new method, called the {parallel replica method} (or 
ParRep), is proposed for efficiently simulating 
equilibrium averages in this setting.

Markov chains are widely used to model physical systems. 
In computational statistical physics -- the main setting 
for this article -- Markov 
chains are used to understand macroscopic properties of 
matter, starting from a mesoscopic or microscopic description. 
Equilibrium averages then correspond to bulk properties of 
the physical system under consideration, like average 
density or internal energy. A popular class of such models 
are the Markov State Models~\cite{chodera,prinz,schutte}. 
Markov chains also arise as 
time discretizations of continuous time models  
like the Langevin dynamics~\cite{tonybook}, a popular 
stochastic model for molecular dynamics. For examples 
of Markov chain models not obtained from an 
underlying continuous time dynamics, see 
for example~\cite{bovier,scoppola}. 
It should be emphasized that the discrete in 
time setting is generic -- even if the underlying 
model is continuous in time, what must be simulated 
in practice is a time-discretized version.

In computational statistical physics, metastability arises 
from entropic barriers, which are bottlenecks in state 
space, as well as energetic barriers, which are 
regions separating metastable states through which 
crossings are unlikely (due to, for example, 
high energy saddle points in a potential energy 
landscape separating the states). See Figures~1--2 
for simple examples of entropic and energetic 
barriers.

The method proposed here is closely related 
to a recently proposed algorithm~\cite{aristoff}, 
also called ParRep, for efficient simulation 
of metastable Markov chains on a coarsened state space. 
That algorithm can be considered an adaptation of 
A.F. Voter's parallel replica dynamics~\cite{voter}
to a discrete time setting. (For a mathematical analysis of 
A.F. Voter's original algorithm, see~\cite{lebris}.) 
ParRep was shown to be 
consistent with an analysis based on  
quasistationary distributions (QSDs), or local 
equilibria associated with each metastable set. 
ParRep uses parallel processing 
to explore phase space more efficiently in real 
time. A cost of the parallelization is that 
only a {\em coarse} version of the Markov chain dynamics, defined on 
the original state space modulo the collection of metastable sets, 
is obtained.  
In this article it is shown that a simple 
modification of the ParRep algorithm of~\cite{aristoff} 
nonetheless allows for  
computation of equilibrium averages of the 
original, {\em uncoarsened} Markov chain.

\begin{figure}
\vskip-140pt
\centerline{\includegraphics[scale=0.65]{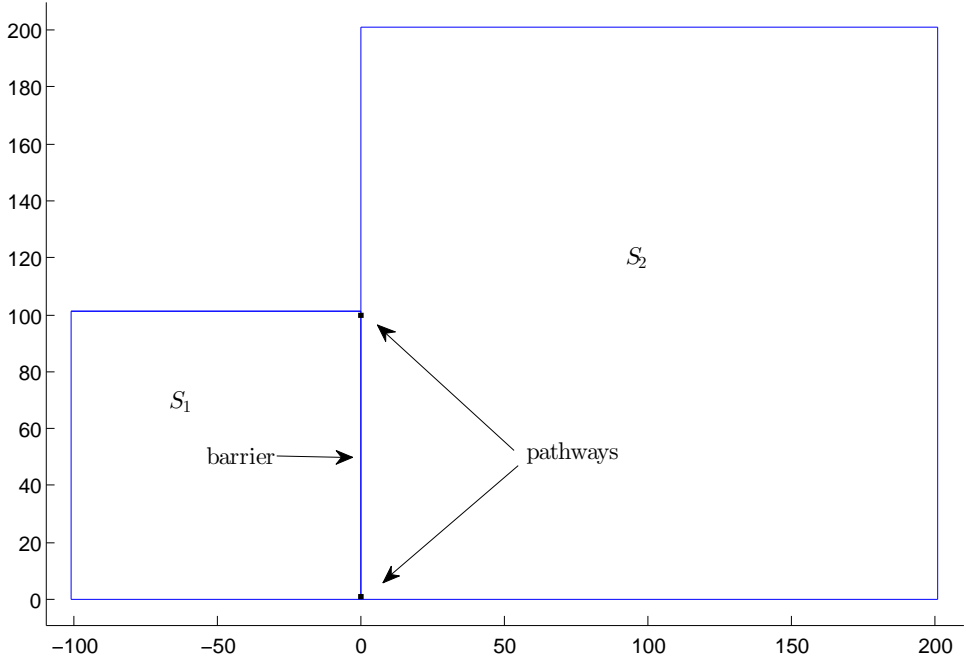}}
\vskip-130pt
\caption{A random walk $X_n$ on state space 
$\{-1,-2,\ldots,-100\}^2 \cup \{1,2,\ldots,200\}^2$ with an 
entropic barrier. At each step, a direction up, down, left or 
right is selected at random, each with probability $1/4$. 
Then $X_n$ moves one unit in this direction, 
provided this does not result in crossing a barrier, i.e., one of the 
edges of the two boxes pictured. The walk can cross from the 
left box to the right box only through the narrow pathways indicated. 
The metastable sets are ${\mathcal S} = \{S_1,S_2\}$.}
\end{figure}

\begin{figure}
\vskip-130pt
\centerline{\includegraphics[scale=0.65]{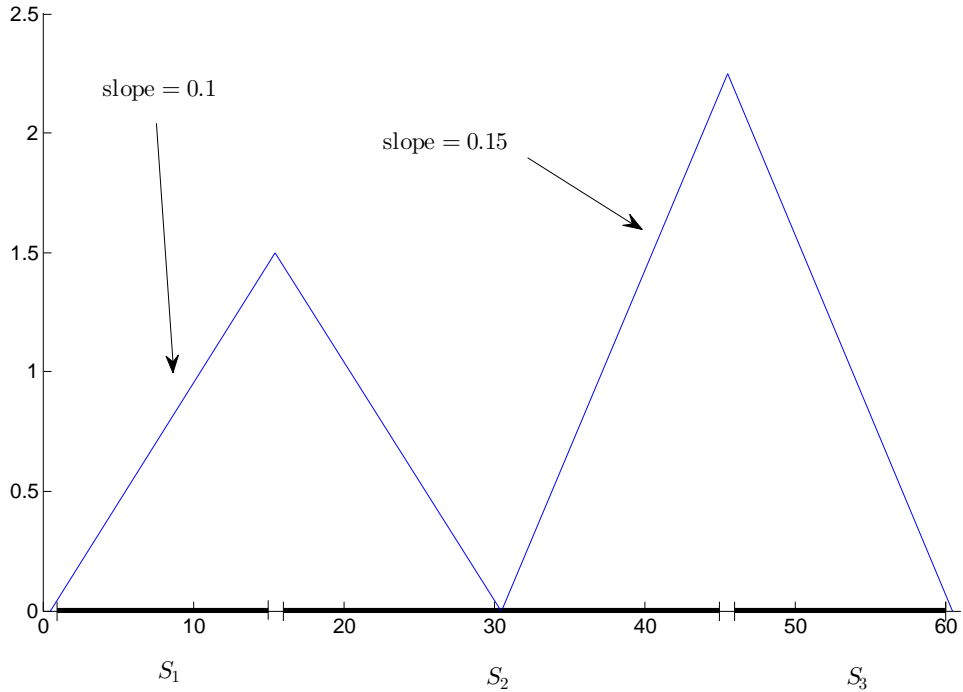}}
\vskip-130pt
\caption{A random walk $X_n$ on state space 
$\{1,2,\ldots,60\}$ with energy barriers. 
The random walk moves one unit left or right according to a 
biased coin flip: If $X_n = x$ and the slope 
of the pictured graph at $x$ is $m$, then with probability 
$1/2 + m$, $X_{n+1} = \max\{x-1,1\}$, and 
with probability $1/2-m$, $X_{n+1} = \min\{x+1,60\}$. The 
metastable sets are ${\mathcal S} = \{S_1,S_2,S_3\}$.}
\end{figure}

The ParRep algorithm proposed here is very general. It can 
be applied to {any} Markov chain, and gains in 
efficiency can be expected when the chain is 
metastable and the metastable sets can be properly 
identified (either a priori or on the fly). In particular, 
it can be applied to metastable Markov chains with 
both energetic and entropic barriers, and no assumptions 
about barrier heights, temperature or reversibility 
are required. 
While there exist many methods for sampling from 
a distribution, most methods, particularly 
in Markov chain Monte Carlo~\cite{rubinstein}, rely on a priori knowledge of  
relative probabilities of the distribution. In contrast with these methods, 
ParRep does not require {\em any} information about the 
equilibrium distribution of the Markov chain.

The article is organized as follows. Section~\ref{sec:QSD} 
defines the QSD 
and notation used throughout.  
Section~\ref{sec:ParRep} introduces the ParRep 
algorithm for computing equilibrium averages 
(Algorithm~\ref{alg1}). In Section~\ref{sec:numerics}, 
consistency of the algorithm is demonstrated on 
the simple models pictured in Figures~1 and~2. 
A proof of consistency in an idealized setting is given 
in the Appendix.
Some concluding remarks are made in Section~\ref{sec:conclude}.

\section{Notation and the quasistationary distribution}\label{sec:QSD}
Throughout, $(X_n)_{n\ge 0}$ is a time homogeneous 
Markov chain on a standard Borel state space, and 
${\mathbb P}_\xi$ is the associated measure when $X_0 \sim \xi$, 
where $\sim$ denotes equality in law. 
All sets and functions are assumed measurable without explicit mention. 
The collection of metastable sets will be written ${\mathcal S}$, 
with elements of ${\mathcal S}$ denoted by $S$. 
Formally, ${\mathcal S}$ is simply a set of disjoint subsets of state space.

\begin{definition}\label{D1}
A probability measure $\nu$ with support in $S$ is 
called a quasistationary distribution (QSD) if for 
all $n\ge 0$ and all $A \subset S$, 
\begin{equation*}
\nu(A) = {\mathbb P}_\nu\left(X_n \in A\,|\,X_1 \in S,\ldots, X_{n} \in S \right).
\end{equation*}
\end{definition}
That is, if $X_n \sim \nu$, then conditionally on $X_{n+1} \in S$,  
$X_{n+1} \sim \nu$. It is not hard to check that, if for 
every probability measure $\xi$ 
supported in $S$ and every $A \subset S$,
\begin{equation}\label{QSD2}
  \nu(A) = \lim_{n\to \infty} 
{\mathbb P}_\xi\left(X_n \in A\,|\,X_1 \in S,\ldots,X_{n} \in S\right),
\end{equation}
then $\nu$ is the unique QSD in $S$. Informally, if~\eqref{QSD2} 
holds, then $(X_n)_{n\ge 0}$ is close to $\nu$ whenever it spends 
a sufficiently long time in $S$ without leaving. 
Of course $\nu$ depends on $S$, but 
this will not be indicated explicitly.

\section{The ParRep algorithm}\label{sec:ParRep}

Let $(X_n)_{n\ge 0}$ be ergodic with equilibrium measure $\mu$, and fix a 
bounded real-valued function $f$ defined on state space.  
The output of ParRep is an estimate of the average of $f$ with respect 
to $\mu$. The algorithm requires existence of a unique QSD in 
each metastable set, so it is assumed for each $S \in {\mathcal S}$ 
there is a unique $\nu$ satisfying~\eqref{QSD2}. 
This assumption holds under very general mixing conditions; 
see~\cite{aristoff}.

The user-chosen parameters of the algorithm are the 
number of replicas, $N$; the decorrelation 
and dephasing times, $T_{corr}$ and $T_{phase}$; 
and a polling time, $T_{poll}$. The parameters 
$T_{corr}$ and $T_{phase}$ are closely related to the 
time needed to reach the QSD; both may depend 
on $S \in {\mathcal S}$. To emphasize this, sometimes 
$T_{corr}(S)$ or $T_{phase}(S)$ are written.  
The parameter $T_{poll}$ is a polling time at which the 
parallel replicas resynchronize. See below
for further discussion. 
\begin{algorithm}\label{alg1}
Set the simulation clock to zero: $T_{sim} = 0$, 
and set $f_{sim} = 0$. Then iterate the following:
\begin{itemize}
\vskip10pt
\item {\bf \em Decorrelation Step.} 

\noindent Evolve $(X_n)_{n\ge 0}$ from 
time $n = T_{sim}$ until time $n = \sigma$, where 
$\sigma$ is the smallest number $n \ge T_{sim}+T_{corr}-1$ 
such that there exists $S \in {\mathcal S}$ with 
$X_n \in S,\, X_{n-1} \in S,\ldots, X_{n-T_{corr}+1} \in S$.
Meanwhile, update
\begin{equation*}
f_{sim} = f_{sim} + \sum_{n=T_{sim}+1}^{\sigma} f(X_n).
\end{equation*}
Then set $T_{sim}= \sigma$ and proceed to the Dephasing Step, 
with $S$ now the metastable state having $X_\sigma \in S$.
\vskip10pt
\item {\bf \em Dephasing Step.} 

\noindent Generate $N$ independent 
samples, $x_1, \ldots, x_N$, 
of the QSD $\nu$ in $S$. Then 
proceed to the Parallel Step. 
\vskip10pt
\item {\bf \em Parallel Step.} 

\noindent (i) Set $M = 1$ and 
$\tau_{acc} = 0$. Let $(X_n^1)_{n\ge 0},\ldots,(X_n^N)_{n\ge 0}$ 
be replicas of $(X_n)_{n\ge 0}$, that is, Markov chains with the same law as 
$(X_n)_{n\ge 0}$ which are independent of 
$(X_n)_{n\ge 0}$ and one another. Set $X_0^1 = x_1$,..., $X_0^N = x_N$.

\vskip5pt
\noindent (ii) Evolve all the replicas 
from time $n = (M-1)T_{poll}$ 
to time $n = MT_{poll}$. 
\vskip5pt

\noindent (iii) If none 
of the replicas leave $S$ during this time, 
update 
\begin{align}
&f_{sim} = f_{sim} + \sum_{i=1}^N \sum_{j=(M-1)T_{poll}+1}^{MT_{poll}} f(X_j^i),\label{sum1} \\
&\tau_{acc} = \tau_{acc} + NT_{poll}, \notag\\
&M = M+1, \notag
\end{align}
and return to (ii) above. Otherwise, let $K$ be the smallest 
number such that $(X_n^K)_{n\ge 0}$ leaves $S$ 
during this time, let $\tau^K \in [(M-1)T_{poll}+1,MT_{poll}]$ 
be the corresponding first exit time, and 
update 
\begin{align}
&f_{sim} = f_{sim} +  \sum_{i=1}^{K-1} \sum_{j=(M-1)T_{poll}+1}^{MT_{poll}}  
f(X_j^i)+ \sum_{j=(M-1)T_{poll}+1}^{\tau^K} f(X_j^K), \label{sum2} \\
&\tau_{acc} = \tau_{acc} + (K-1)T_{poll} + (\tau^K-(M-1)T_{poll}).\notag
\end{align}
Then update $T_{sim} = T_{sim} + \tau_{acc}$, 
set $X_{T_{sim}} = X_{acc} := X_{\tau^K}^K$, and return to the Decorrelation Step. 
\end{itemize}
\end{algorithm} 
See Figure~3 for an illustration of the Parallel Step. 
The key quantity in the algorithm is 
the running average $f_{sim}/T_{sim}$, which is 
an estimate of the average of $f$ with respect to the 
equilibrium measure $\mu$:
\begin{equation*}
 \frac{f_{sim}}{T_{sim}} \approx \int f\,d\mu.
\end{equation*}
Some remarks on Algorithm~\ref{alg1} are in order. 

\begin{figure}
\vskip-60pt
\centerline{\includegraphics[scale=0.62]{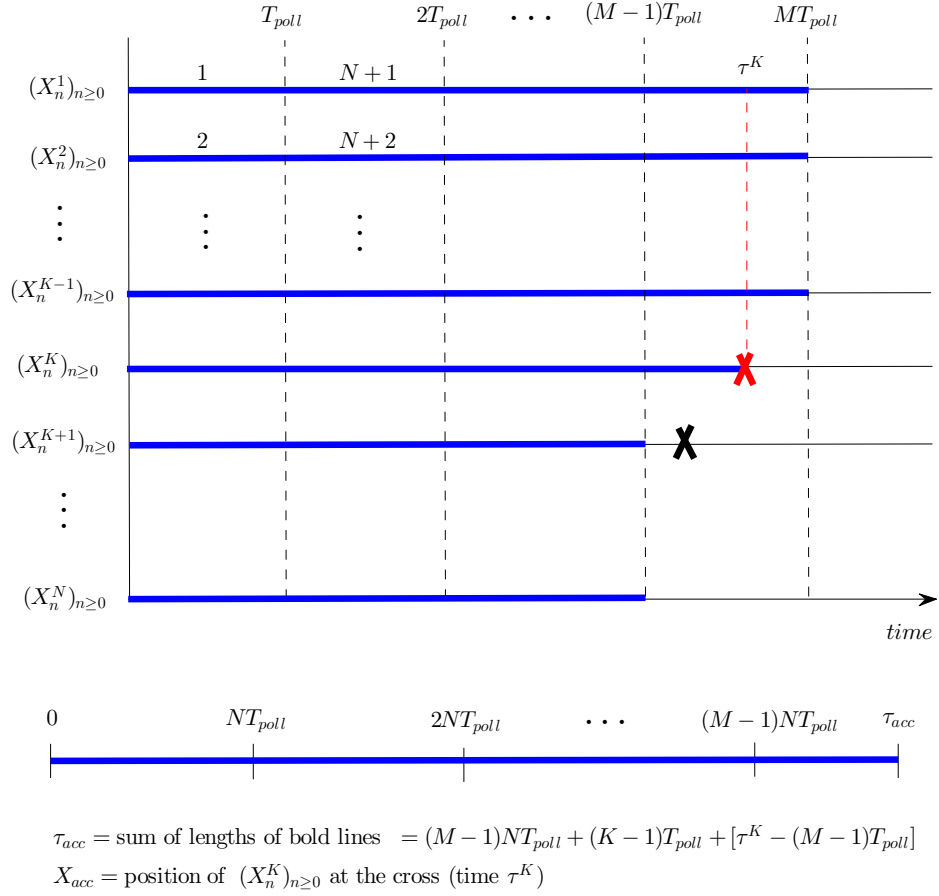}}
\vskip-70pt
\caption{Visualization of the Parallel Step of Algorithm~\ref{alg1}. 
The crosses represent exits from $S$. After $M$ loops internal to the Parallel Step, 
two of the replicas leave $S$, with $(X_n^K)_{n\ge 0}$, 
the one among these having the smallest index $K$, leaving at time $\tau^K$. 
The trajectories of all the replicas can be concatenated into a single 
long trajectory of length $\tau_{acc}$. This single long trajectory is obtained 
by running through the columns of width $T_{poll}$ from 
top to bottom, starting at the far left, in 
the order $1,2,\ldots,N+1,N+2,\ldots$ indicated. The time marginals 
of this long trajectory (except its right endpoint) 
are all distributed according to the QSD in $S$.}
\end{figure}

\vskip10pt
\begin{itemize}
 \item {\bf The Decorrelation Step.} The purpose of the 
Decorrelation Step is to reach the QSD in some 
metastable set. Indeed, the Decorrelation Step 
terminates exactly when $(X_n)_{n \ge 0}$ has spent 
$T_{corr}$ consecutive time steps in some 
metastable set $S$ -- so 
the position of $(X_n)_{n \ge 0}$ 
at the end of the Decorrelation Step can be 
considered an approximate sample from $\nu$, the QSD in $S$. 
The error in this approximation is controlled by 
the parameter $T_{corr}$. Larger values of 
$T_{corr}$ lead to increased accuracy but 
lessened efficiency; see the numerical tests in 
Section~\ref{sec:numerics} below, in particular  
Figures~4 and~6. During 
the Decorrelation Step, the dynamics of $(X_n)_{n\ge 0}$ 
is exact, so the contribution to $f_{sim}$ 
from the Decorrelation Step is exact. 
\vskip10pt

\item {\bf The Dephasing Step.} The Dephasing Step 
requires sampling $N$ iid copies of the QSD in $S$, where $S$ is
the metastable set from the end of the Decorrelation Step. 
The practitioner has flexibility in sampling these points. 
Essentially, one has to sample $N$ endpoints of trajectories 
of $(X_n)_{n\ge 0}$ that have remained in $S$ 
for a long enough time, with this time being 
controlled by the parameter $T_{phase}$. For example, the Dephasing Step can
be done with rejection sampling, keeping trajectories 
which have remained in $S$ for time $T_{phase}$. 
Alternatively, the QSD samples may be obtained via 
techniques related to the Fleming-Viot process; 
for details see~\cite{aristoff} and~\cite{binder}. 
This technique can be summarized as follows: 
$N$ replicas of $(X_n)_{n\ge 0}$, all starting in $S$, are 
independently evolved 
until one or several leave $S$; then  
each replica which left $S$ is restarted from the current position 
of another replica still inside $S$, 
chosen uniformly at random. After time $T_{phase}$ this 
procedure stops and the current positions of the 
replicas are used as the $N$ required samples of $\nu$. 

Under mild mixing conditions, convergence to the QSD 
is very fast. More precisely, 
the limit in the right hand side of~\eqref{QSD2} converges to $\nu$ 
geometrically fast in total variation norm~\cite{aristoff}. 
An analysis of the error associated with not 
exactly reaching the QSD will be the focus of 
another work. 
For an analysis of the error associated with 
not reaching the QSD in the original 
continuous-in-time version 
of the algorithm, see~\cite{gideon}.
In the metastable setting considered here, 
the average time to (approximately) reach the QSD 
in $S$ is assumed much smaller than the average 
time, starting at the QSD, to leave $S$. 
Indeed, this assumption can be considered 
the very definition of metastability. Gains 
in efficiency in ParRep are limited by the 
degree of metastability; see~\cite{binder} 
and the discussion in Section~\ref{sec:numerics} 
below. 

It is emphasized that $f_{sim}$ and $T_{sim}$ are left unchanged 
during the Dephasing Step. Contributions to $f_{sim}$ and $T_{sim}$ 
come only from the Decorrelation and Parallel Steps. 

\vskip10pt
\item {\bf The Parallel Step.} The purpose of the Parallel Step is twofold. First, it 
simulates an exit event from $S$, the metastable 
set from the end of the Decorrelation Step, starting from the QSD in $S$. 
This is consistent with the exit event that would have 
been observed if, in the Decorrelation Step, 
$(X_n)_{n\ge 0}$ had been allowed to continue evolving 
until leaving $S$:
\begin{theorem}\label{T0b}{\em (Proposition~4.5 of~\cite{aristoff}.)} Suppose the QSD sampling 
in the Dephasing Step of Algorithm~\ref{alg1} is {\em exact}. 
Then in the Parallel Step, $(\tau_{acc},X_{acc}) \sim (\tau,X_\tau)$, 
where $X_0 \sim \nu$, with $\nu$ the QSD in $S$ and $\tau := \min\{n\ge 0\,:\,X_n \notin S\}$. 
\end{theorem}
The gain in efficiency 
in ParRep, compared to direct serial simulation, 
comes from the use of parallel processing 
in the Parallel Step. The wall-clock time speedup -- 
the ratio of average serial simulation time to the ParRep simulation 
time of the exit event -- scales like $N$, though the 
gain in efficiency in ParRep as a whole depends 
also on $T_{corr}$, $T_{phase}$ and the degree of metastability 
of the sets in ${\mathcal S}$. 

Second, the Parallel Step includes a contribution 
to $f_{sim}$. 
As the fine scale dynamics of $(X_n)_{n \ge 0}$ in 
$S$ are not retained in the Parallel step, this 
contribution is not exact. It is, however, consistent {\em 
on average}, which is sufficient for the computation 
of equilibrium averages. This can be understood as follows. 
Concatenate all the trajectories of all the replicas into 
a single long trajectory by following the procedure 
indicated in Figure~3. The resulting trajectory 
has a probability law that is of course different 
from that of $(X_n)_{n \ge 0}$ starting from the QSD $\nu$ in $S$. 
However, in light of Definition~\ref{D1}, 
the {\em time marginals} of this trajectory (except the right endpoint) 
are all distributed according to $\nu$. Moreover, from 
Theorem~\ref{T0b}, the total length of this concatenated 
trajectory has the same law as 
that of $(X_n)_{n \ge 0}$ started from the QSD in $S$ and 
stopped at the first exit time from $S$. So by linearity 
of expectation, the 
contribution to $f_{sim}$ from the Parallel Step 
is consistent on average. See the Appendix
for proofs of these statements in an idealized setting.
\vskip10pt
\item {\bf Other remarks.} The parameter $T_{poll}$ is a 
polling time at which the possibly asynchronous parallel 
processors in the Parallel Step resynchronize. For the 
Parallel Step to be finished correctly, one has to wait 
until the first $K$ processors have completed $MT_{poll}$ time 
steps. If the processors are nearly synchronous or communication 
between them is cheap, one can take $T_{poll} = 1$. 

The metastable sets ${\mathcal S}$ need not 
be known a priori. In many applications, they can 
be identified on the fly; for example, when the 
metastable sets are the basins of attraction of a 
potential energy, they can be found efficiently 
on the fly by gradient descent. The reader is 
referred to~\cite{binder} as well as~\cite{aristoff} 
and references therein for examples of successful 
applications of related versions of ParRep in this setting.
\end{itemize}
\vskip10pt

\section{Numerical tests}\label{sec:numerics}

\subsection{Example 1: Entropic barrier}
Consider the Markov chain from Figure 1 
on state space $\{-1,-2,\ldots,-100\}^2 \cup \{1,2,\ldots,200\}^2$. 
The Markov chain evolves according to a random walk: at each 
time step it moves one unit up, down, left or 
right each with probability $1/4$, 
provided the result is inside state space; if not, 
the move is rejected and the position stays the same. 
There is one exception: If the current position is 
$(-1,1)$ or $(-1,100)$ and a move to the right 
is proposed, then the next position is $(1,1)$ or 
$(1,100)$, respectively; and if the current position is 
$(1,1)$ or $(1,100)$ and a move to the left 
is proposed, then the next position is $(-1,1)$ or 
$(-1,100)$, respectively. This 
Markov chain is ergodic with respect 
to the uniform distribution $\mu_{unif}$ on state space. 

\begin{figure}
\vskip-110pt
\centerline{\includegraphics[scale=0.75]{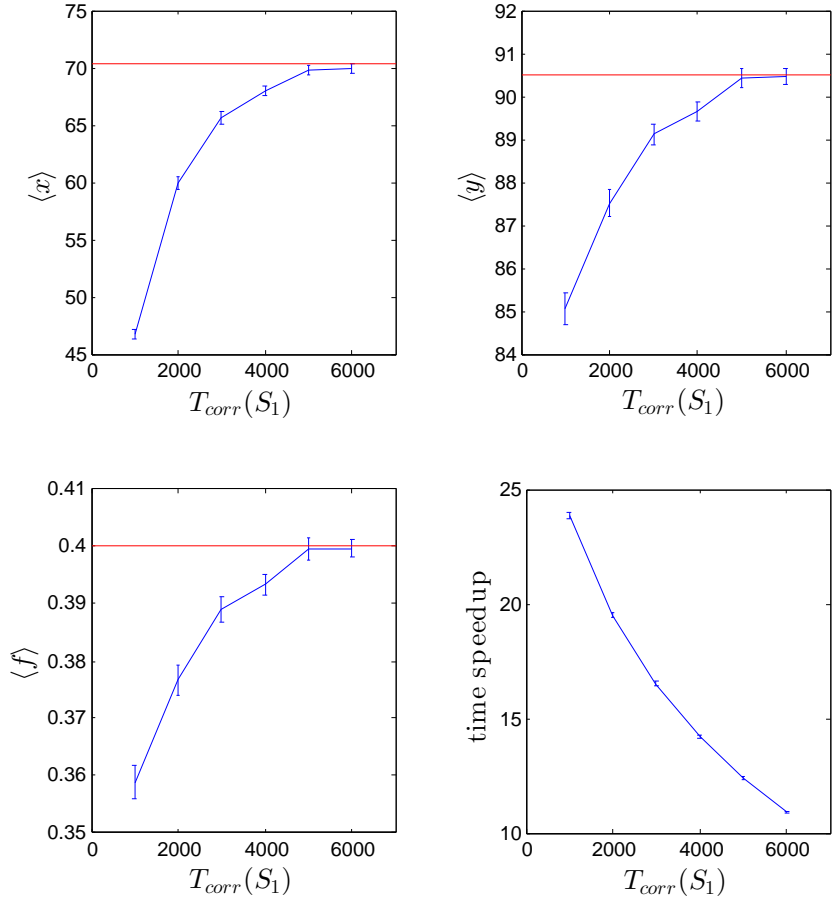}}
\vskip-110pt
\caption{Equilibrium average values 
and (average) time speedup vs. $T_{corr}(S_1)$ in ParRep simulations 
of the Markov chain from Example 1, with $N = 100$. The straight lines correspond 
to exact values. ParRep simulations were stopped when $T_{sim}$ 
first exceeded $5 \times 10^9$, and error bars are standard deviations 
obtained from $100$ independent trials.}
\end{figure}

\begin{figure}
\vskip-110pt
\centerline{\includegraphics[scale=0.75]{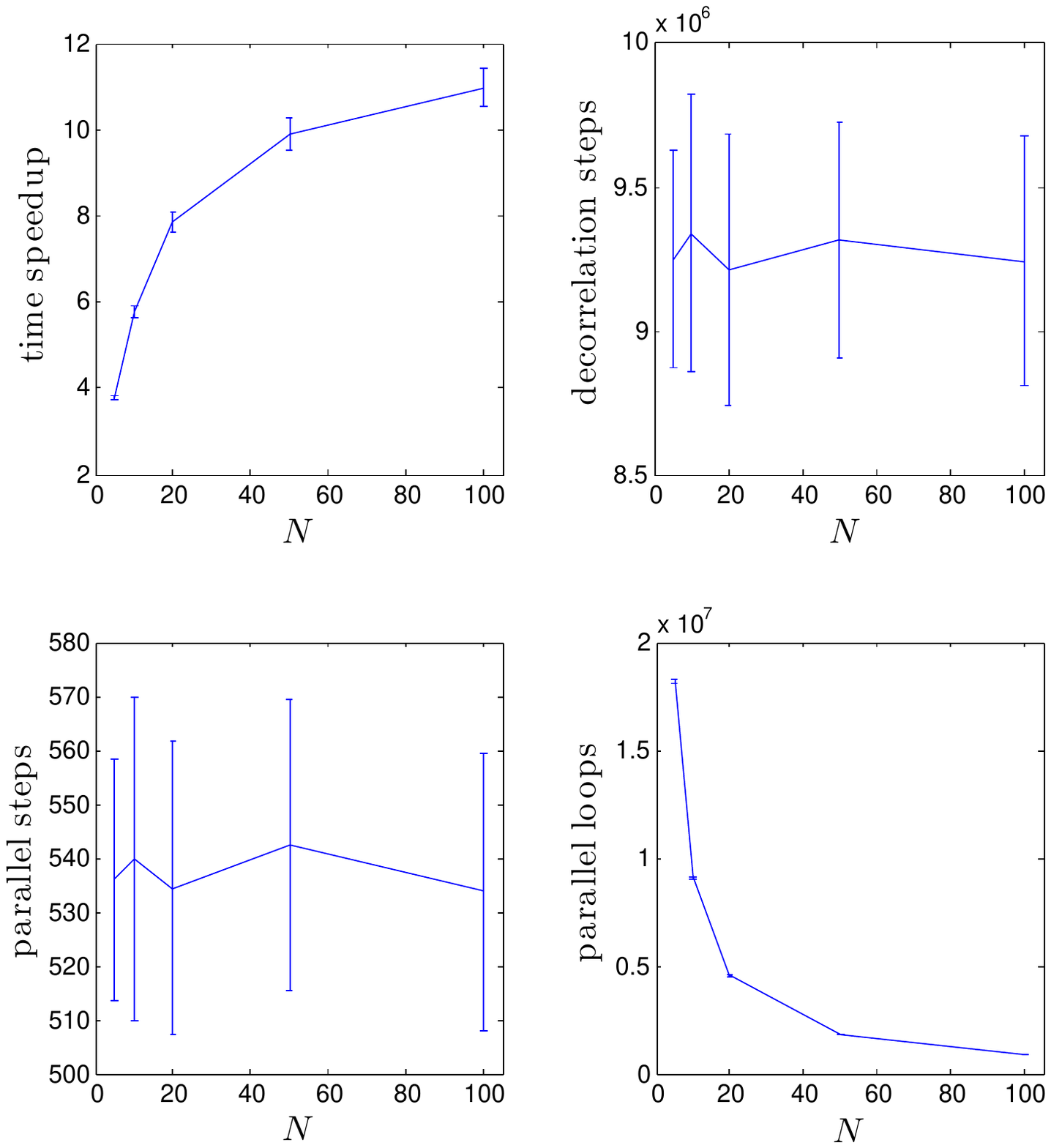}}
\vskip-110pt
\caption{Average time speedup factor, 
number of decorrelation steps, number of 
parallel steps, and number of parallel loops in ParRep simulations 
of the Markov chain from Example 1, with $T_{corr}(S_1) = 6000$. ParRep simulations were stopped when $T_{sim}$ 
first exceeded $10^8$, and error bars are standard deviations 
obtained from $100$ independent trials.}
\end{figure}

ParRep was performed on this system 
with $S_1:= \{-1,-2,\ldots,-100\}^2$, 
$S_2:= \{1,2,\ldots,200\}^2$, and ${\mathcal S} = \{S_1,S_2\}$. 
Parameters were always chosen so that 
$T_{corr} = T_{phase}$ and $T_{corr}(S_2) = 4T_{corr}(S_1)$, and 
QSD samples from the Dephasing Step were obtained using the Fleming-Viot-based 
technique described above. 

With $N = 100$ replicas and various values of $T_{corr}(S_1)$, 
ParRep was used to obtain average $x$- and $y$-coordinates 
with respect to $\mu_{unif}$ as well as the $\mu_{unif}$-probability 
to be in the upper half of the right hand side box, denoted by:
\begin{align*}
&\langle x \rangle := \int x \,d\mu_{unif}, \qquad
\langle y \rangle := \int y \,d\mu_{unif},\\
&\langle f \rangle := \int 1_{y \in [101,200]} \,d\mu_{unif}.
\end{align*}
Here $1_A$ denotes the indicator function of $A$. See Figure~4. 
Also computed was the average {\em time speedup}: namely,  
$T_{sim}$ divided by the ``wall clock time,'' defined as follows.  
Like $T_{sim}$, the wall clock time stars at zero. It 
increases by $1$ during each time step of $(X_n)_{n\ge 0}$ in the 
Decorrelation Step (consistent with $T_{sim}$), 
while it increases by $MT_{poll}$ in the Parallel Step (unlike $T_{sim}$, 
which increases by $\tau_{acc}$). The wall clock time also 
increases by $T_{phase}$ during the dephasing step 
(where $T_{sim}$ does not increase at all). Informally, 
the wall clock time corresponds to true clock time 
in an idealized setting where all the processors 
always compute one time step of $(X_n)_{n\ge 0}$ in exactly $1$ 
unit of time, and communication between processors takes zero time. 
As $T_{corr}$ increases, the time speedup decreases, but accuracy increases.

\begin{figure}
\vskip-110pt
\centerline{\includegraphics[scale=0.75]{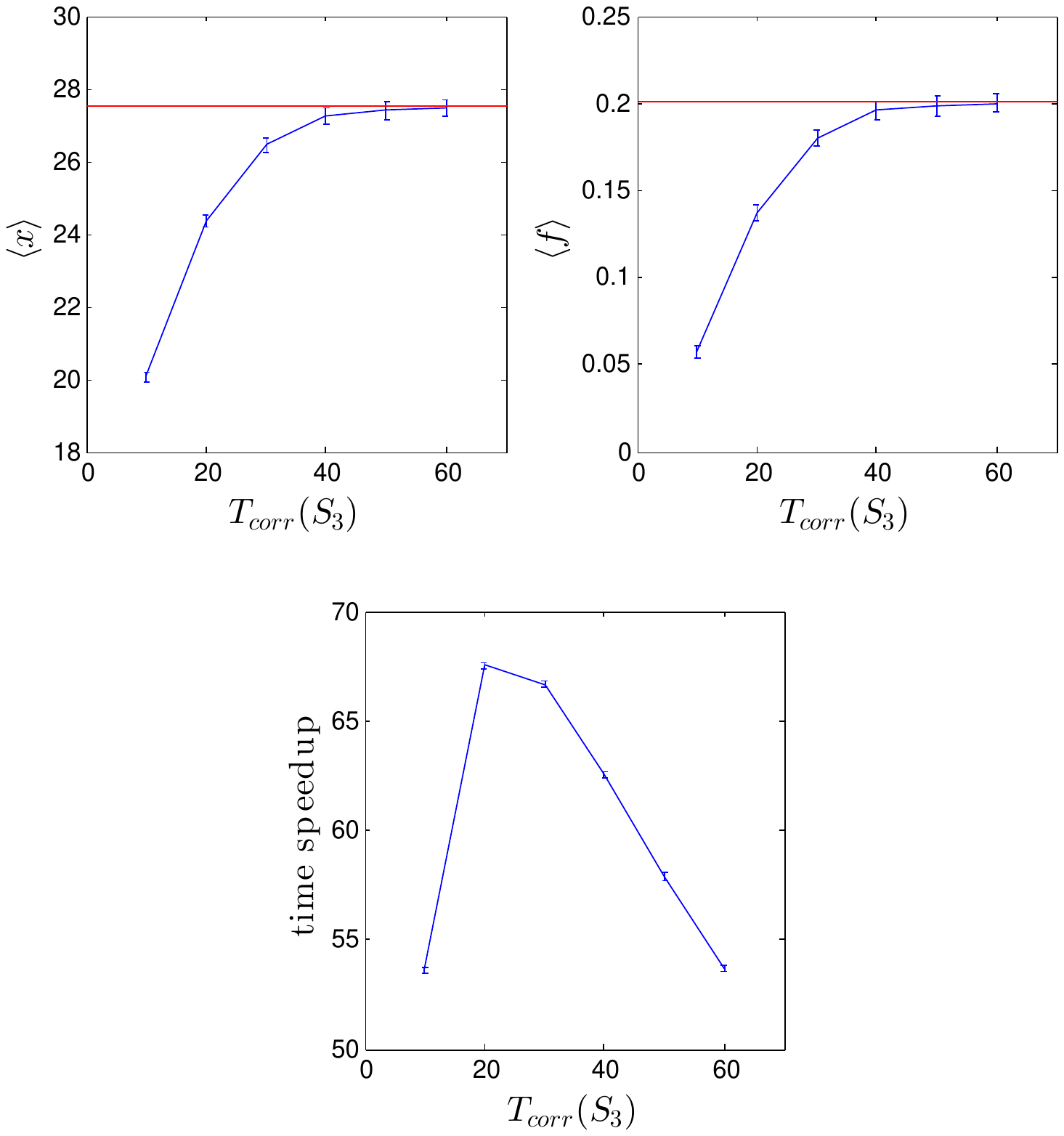}}
\vskip-110pt
\caption{Equilibrium average values 
and (average) 
time speedup vs. $T_{corr}(S_3)$ in ParRep simulations 
of the Markov chain from Example 2, with $N = 100$. The straight lines correspond 
to exact values. ParRep simulations were stopped when $T_{sim}$ 
first exceeded $10^9$, and error bars are standard deviations 
obtained from $100$ independent trials. For the smallest value of 
$T_{corr}(S_3)$, the Markov chain is typically close to the edges of 
$S_1$, $S_2$ or $S_3$, which results in shorter parallel steps and thus 
a smaller time speedup.}
\end{figure}

\begin{figure}
\vskip-110pt
\centerline{\includegraphics[scale=0.75]{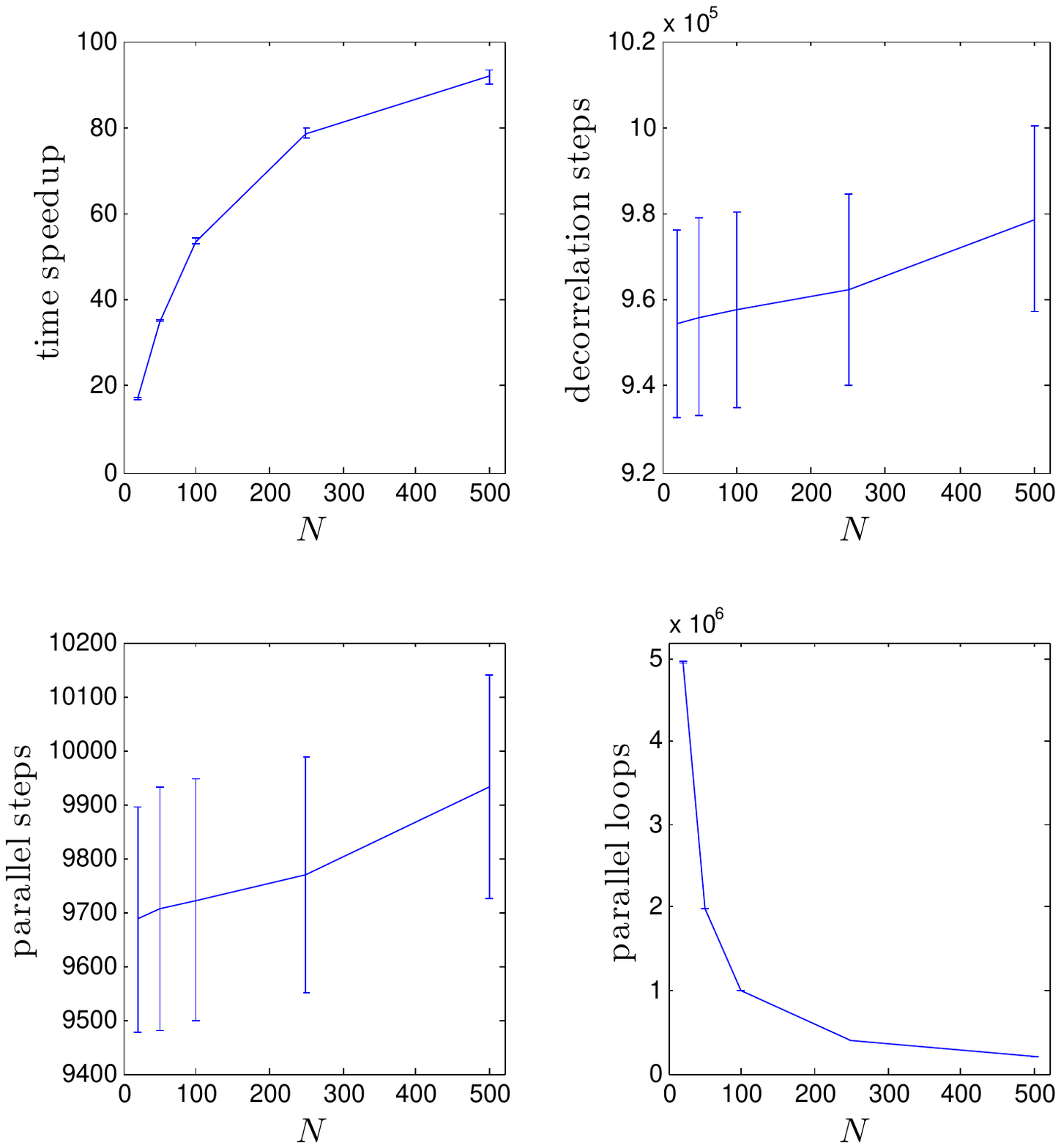}}
\vskip-110pt
\caption{Average time speedup factor, 
number of decorrelation steps, number of 
parallel steps, and number of parallel loops in ParRep simulations 
of the Markov chain from Example 2, when $T_{corr}(S_3) = 60$. ParRep simulations were stopped when $T_{sim}$ 
first exceeded $10^8$, and error bars are standard deviations 
obtained from $100$ independent trials.}
\end{figure}

Figure~5 shows the dependence of time speedup on the number 
of replicas, $N$, when $T_{corr}(S_1) = 6000$. To 
illuminate the dependence of time speedup on $N$, Figure~5 also includes the average 
(total) number of decorrelation steps, 
parallel steps, and parallel loops (i.e., loops internal to the 
parallel step -- in the notation of Algorithm~\ref{alg1}, there 
are $M$ loops internal to the parallel step). As $N$ increases,
the number of parallel loops decreases sharply, 
while the number of parallel steps and decorrelation steps remain 
nearly constant. Thus, with increasing $N$ the wall clock time 
spent in the parallel step falls quickly. The time speedup, however, is 
limited by the wall clock time spent in the decorrelation step, and   
so it levels off with increasing $N$. 
The value of $N$ at which this leveling off occurs depends on the degree 
of metastability in the problem, or slightly more precisely, 
the ratios, over all $S \in {\mathcal S}$, 
of the time scale for leaving $S$ to 
the time scale for reaching the QSD in $S$. In the limit as this 
ratio approaches infinity, the time speedup grows like $N$. See~\cite{binder} 
for a discussion of this issue in a continuous time version of ParRep.

\subsection{Example 2: Energetic barrier}
Consider the Markov chain from Figure~2 on state space $\{1,\ldots,60\}$. 
The Markov chain evolves according to a biased random walk: 
If $X_n = x$, then with probability 
$p_x$, $X_{n+1} = \max\{1,x-1\}$, while with probability $1-p_x$, 
$X_{n+1} = \min\{60,x+1\}$. Here, 
\begin{equation*}
 p_x := \begin{cases} 0.6, & x \in \{1,\ldots,15\} \\ 
        0.4, & x \in \{16,\ldots,30\} \\
        0.65, & x \in \{31,\ldots,45\}\\
        0.35, & x \in \{46,\ldots,60\}.
       \end{cases}
\end{equation*}
The equilibrium distribution $\mu_{bias}$ of this Markov 
chain can be explicitly computed. 

ParRep simulations were performed on this system 
with $S_1:= \{1,\ldots,15\}$, 
$S_2:= \{16,\ldots,45\}$, $S_3 = \{46,\ldots,60\}$ 
and ${\mathcal S} = \{S_1,S_2,S_3\}$. 
Parameters were always chosen so that 
$T_{corr} = T_{phase}$ and $T_{corr}(S_1) = T_{corr}(S_2) = \frac{3}{2}T_{corr}(S_3)$, 
and QSD samples from the Dephasing Step were again obtained using the Fleming-Viot-based 
technique.  

With $N = 100$ replicas and various values of $T_{corr}(S_3)$, 
ParRep was used to obtain the average $x$-coordinate 
with respect to $\mu_{bias}$ as well as the $\mu_{bias}$-probability 
to be in the right half of the interval, denoted by:
\begin{align*}
&\langle x \rangle := \int x \,d\mu_{bias},\\
&\langle f \rangle := \int 1_{x \in [31,60]} \,d\mu_{bias}.
\end{align*}
Also computed was the {time speedup}, defined exactly as above. 
Simulations were stopped when $T_{sim}$ first exceeded $10^9$. 
See Figure 6. Again, accuracy increases with $T_{corr}$, but the 
time speedup decreases with $T_{corr}$. 

Figure~7 shows the dependence of time speedup on the number 
of replicas, $N$, when $T_{corr}(S_3) = 60$. 
Also plotted are the average number of decorrelation steps, 
parallel steps, and parallel loops. The results are similar to Example~1, 
though the degree of metastability and time speedup are much larger.

\section{Conclusion}\label{sec:conclude}

A new algorithm, ParRep, for computing equilibrium averages of 
Markov chains is presented. The algorithm requires no knowledge 
about the equilibrium distribution of the Markov chain. Gains 
in efficiency are obtained by asynchronous parallel processing. 
For these gains to be achievable in practice, the Markov chain 
must possess some metastable sets. These 
sets need not be known a priori, but they should be 
identifiable on the fly; for example, in many applications in 
computational chemistry, the metastable sets can be
basins of attraction of a potential energy, 
identified on the fly by gradient descent. 
When metastable sets are present, the gains in efficiency 
are limited by the degree of metastability. See~\cite{binder} 
for a discussion and an application of a related version 
of ParRep in this setting.

Applications in computational chemistry seem numerous. 
Nearly all popular stochastic models of molecular dynamics 
are Markovian. Even when these models are continuous in time, 
to actually simulate the models a time discretization 
is required and the result is a Markov chain. Generically, 
these models have many metastable sets associated with 
different geometric arrangements of atoms at distinct
local minima of the potential energy or free energy. Many times 
the equilibrium distributions of these models are unknown -- for 
example if external forces are present -- yet it is still of 
great interest to sample equilibrium. Because of metastability 
it is often impractical or impossible 
to sample equilibrium with direct simulation. ParRep may 
put such computations within reach.

\section*{Appendix}\label{appendix}

In this Appendix, consistency of ParRep is proved in an 
idealized setting. 

\subsection{Idealized setting, assumptions, and main result}
Recall $(X_n)_{n\ge 0}$ is a Markov chain 
on a standard Borel state space $(\Omega, {\mathcal F})$. 
The collect ${\mathcal S} \subset {\mathcal F}$ of 
disjoint sets is assumed finite, with a 
unique QSD $\nu$ associated to each metastable set $S$. All probabilities, 
which may be associated to different spaces and random 
processes or variables, will be denoted by ${\mathbb P}$; the meaning 
will be clear from context. Probabilities 
associated with the initial distribution $\xi$ are 
denoted by ${\mathbb P}_\xi$. 
(If $\xi = \delta_x$ then ${\mathbb P}_x$ is written instead.) 
The corresponding 
expectations are written ${\mathbb E}$, ${\mathbb E}_\xi$, 
or ${\mathbb E}_x$. The norm 
$||\cdot||$ will always be total variation norm.

In all the analysis below, an idealized setting is assumed. 
It is defined by two conditions: 
the QSD is sampled {\em exactly} in the 
Dephasing Step (Idealization~\ref{A0}), 
and the QSD is reached {\em exactly} by time $T_{corr}$ (Idealization~\ref{A1}). 
These are idealizing assumptions  
in the sense that, in practice, the QSD is never exactly reached.    
\begin{idealization}\label{A0}
In the Dephasing Step, the points 
$x_1,\ldots,x_N$ are drawn independently and 
{\em exactly} from the QSD $\nu$ in $S$.
\end{idealization}
\begin{idealization}\label{A1}
For each $S \in {\mathcal S}$ there is 
a time $T_{corr} > 0$ such that, after spending 
$T_{corr}$ consecutive time 
steps in $S$, the Markov chain $(X_n)_{n\ge 0}$ is {\em exactly} 
distributed according to $\nu$. That is, for every 
$S \in {\mathcal S}$
and every $A \in {\mathcal F}$ with $A \subset {S}$, 
\begin{equation*}
 {\mathbb P}(X_{T_{corr}} \in A\,|\, X_1 \in S,\ldots,X_{T_{corr}} \in S) = \nu(A).
\end{equation*}
\end{idealization}
In practice, the word {\em exactly} must be replaced 
with {\em approximately}. The error associated with 
not exactly reaching the QSD will not be studied here. See however~\cite{gideon} 
for an analysis of this error in the continuous 
in time setting. Here, the idealized setting seems necessary 
to connect the ParRep dynamics with those of the original 
Markov chain. The idealizations allow these two 
dynamics to be synchronized after reaching the QSD, 
which is crucial in the analysis below. In particular, 
the analysis here cannot be modified in a simple way 
to allow for inexact convergence to the QSD.

By Idealization~\ref{A1}, at the end of each Decorrelation Step 
$(X_n)_{n\ge 0}$ is distributed {\em exactly} according to the QSD. 
By Idealization~\ref{A0}, the Parallel Step is exact:  
\begin{theorem}(Restated from~\cite{aristoff}.)\label{T0}
Let Idealization~\ref{A0} hold.
Then in the Parallel Step of Algorithm~\ref{alg1}, $(\tau_{acc},X_{acc}) \sim (\tau,X_\tau)$, 
where $X_0 \sim \nu$, with $\nu$ the QSD in $S$ and $\tau := \min\{n\ge 0\,:\,X_n \notin S\}$. 
Moreover, $\tau$ is a geometric random 
variable with parameter $p:= {\mathbb P}_\nu(X_1 \notin S)$, 
and $\tau$ is independent of $X_{\tau}$. 
\end{theorem}
In particular, the first exit time from $S$, starting at the QSD, 
is a geometric random variable which is independent 
of the exit position. This property is crucial for proving consistency of 
the Parallel Step (see~\cite{aristoff}), and will be useful below.

To prove the main result, a form of ergodicity for the original 
Markov chain is required:
\begin{assumption}\label{A2}
 The Markov chain $(X_n)_{n\ge 0}$ is uniformly ergodic: that is, 
 there exists a (unique) probability 
 measure $\mu$ on $(\Omega,{\mathcal F})$ such that 
\begin{equation*}
\lim_{n\to \infty} \,\sup_{\xi}||{\mathbb P}_\xi(X_n \in \cdot) - \mu|| = 0.
\end{equation*}
where the supremum is taken over all probability measures 
$\xi$ on $(\Omega,{\mathcal F})$.
\end{assumption}
Next, a Doeblin-like condition is assumed:
\begin{assumption}\label{A3}
 There exists $\alpha \in (0,1)$, 
 $S \in {\mathcal S}$ 
 with $\mu(S) > 0$,  
 and a probability measure $\lambda$ on 
 $(\Omega, {\mathcal F})$ supported in $S$ 
 such that the following holds: for all 
$x \in S$ and all $C \in {\mathcal F}$ 
 with $C \subset S$,
\begin{equation*}
{\mathbb P}_x(X_1 \in C) \ge \alpha \lambda(C).
\end{equation*}
\end{assumption}
Finally, a lower bound is assumed for escape rates from metastable states.
\begin{assumption}\label{A4}
There exists $\delta > 0$ such that for all $S \in {\mathcal S}$, 
\begin{equation*}
{\mathbb P}_\nu(X_1 \notin S) \ge \delta.
\end{equation*}
\end{assumption}
This simply says that none of the metastable sets are absorbing.
The following is the main result of this Appendix:
\begin{theorem}\label{MainTheorem}
Let Idealizations~\ref{A0}-~\ref{A1} 
and Assumptions~\ref{A2}-~\ref{A4} hold. Then 
for any probability measure $\xi$ on $(\Omega,{\mathcal F})$ 
and any bounded measurable function $f:\Omega \to {\mathbb R}$:
\begin{equation*}
{\mathbb P}_\xi\left(\lim_{T_{sim}\to \infty} \frac{f_{sim}}{T_{sim}} = \int_\Omega f\,d\mu\right) = 1.
\end{equation*}
\end{theorem}
The proof of Theorem~\ref{MainTheorem} is in Section~\ref{sec:Proof} 
below. It is emphasized that {\em Idealizations~\ref{A0}-~\ref{A1} 
and Assumptions~\ref{A2}-~\ref{A4} are assumed 
to hold throughout the remainder of the Appendix.} Furthermore, 
for simplicity it is assumed that $T_{corr}$ is the same for 
each $S \in {\mathcal S}$.

\subsection{Proof of main result}\label{sec:Proof}
The first step in the proof is to show that Theorem~\ref{MainTheorem} 
holds when the number of replicas is $N = 1$ 
(Sections~\ref{sec:N1}-~\ref{sec:EP}). Then 
this will be generalized to any number of replicas 
(Section~\ref{sec:MP}). It is known (see Chapter~7 of~\cite{Douc}) 
that Assumption~\ref{A2} is a 
sufficient condition for the following to hold:
\begin{lemma}\label{L1}
There exists a (unique) measure $\mu$ on 
$(\Omega, {\mathcal F})$ such that 
for all probability measures $\xi$ 
on $(\Omega, {\mathcal F})$ and 
all bounded measurable functions $f: \Omega \to {\mathbb R}$, 
\begin{equation*}
{\mathbb P}_\xi\left(\lim_{n\to \infty}\frac{f(X_0) + \ldots + f(X_{n-1})}{n} = \int_\Omega f \,d\mu\right) = 1.
\end{equation*}
\end{lemma}

\subsubsection{The ParRep process with one replica}\label{sec:N1}

Consider a stochastic process
$({\tilde X}_n)_{n\ge 0}$ which represents the 
underlying process in Algorithm~\ref{alg1} 
when the number of replicas 
is $N=1$. Loosely speaking, $({\tilde X}_n)_{n\ge 0}$ 
evolves like $(X_n)_{n\ge 0}$ in the Decorrelation Step, 
and like $(X_n^1)_{n\ge 0}$ in the Parallel Step (and it 
does not evolve during the Dephasing Step). 
More precisely, $({\tilde X}_n)_{n\ge 0}$ can be 
defined in the following way (writing $S$ 
for a generic element of~${\mathcal S}$):

\begin{itemize}
 \item[1.] If ${\tilde X}_n = x$ and $x\in S$ 
do the following. If ${\tilde X}_j \notin S$ 
for some $j \in \{n-1,n-2,\ldots,\max\{0,n-T_{corr}+1\}\}$, pick 
$x'$ from ${\mathbb P}_x(X_1 \in \cdot)$, let $X_{n+1} = x'$, 
update $n = n+1$, and repeat. Otherwise, 
update $n = n+1$ and proceed to 2.

\item[2.] If ${\tilde X}_n = x$ and $x \in S$, pick 
$z$ from the QSD in $S$, pick $x'$ from 
${\mathbb P}_z(X_1 \in \cdot)$, and let ${\tilde X}_{n+1} = x'$. 
If $x' \notin S$, update $n = n+1$ and return to 1. 
Otherwise, update $n = n+1$ and proceed to 3.

\item[3.] If ${\tilde X}_n = x$ and $x \in S$, pick 
$x'$ from ${\mathbb P}_x(X_1 \in \cdot)$, and let 
${\tilde X}_{n+1} = x'$. If $x' \in S$, update 
$n = n+1$ and repeat. Otherwise, update $n = n+1$ 
and return to 1.
\end{itemize}

Note that $({\tilde X}_n)_{n\ge 0}$ is not Markovian, since 
the next value of the process depends on the history of the 
process. Idealization~\ref{A1}, however, implies 
that $X_n$ and ${\tilde X}_n$ have 
the same law for each $n\ge 0$:
\begin{lemma}\label{L2}
If $X_0 \sim {\tilde X}_0$, then for 
every $n \ge 0$, $X_n \sim {\tilde X}_n$.
\end{lemma}

\subsubsection{An extended Markovian process}

Consider next an extended {\em Markovian} process 
$(Y_n)_{n\ge 0}$ with values in $\Omega \times {\mathbb Z}$,  
such that $(\pi_1(Y_n))_{n\ge 0}$ has the same law 
as $({\tilde X}_n)_{n\ge 0}$, where 
$\pi_i:\Omega 
\times {\mathbb Z} \to \Omega$
is projection onto the $i$th component:
\begin{equation*}
\pi_1(x,t) = x, \quad \pi_2(x,t) = t.
\end{equation*}
Loosely speaking, the second component of 
$(Y_n)_{n\ge 0}$ is a counter indicating how 
many consecutive 
steps the process has spent in a given state $S \in {\mathcal S}$. 
The counter stops at $T_{corr}$, even if it
continues to survive in $S$.
The first component of $(Y_n)_{n\ge 0}$ evolves exactly 
like $(X_n)_{n\ge 0}$, except when the second component 
is $T_{corr}-1$, in which case, starting at a sample 
of the QSD in $S$, the process is evolved  
one time step. It is convenient to describe $(Y_n)_{n\ge 0}$ 
more precisely as follows (writing $S$ for a generic 
element of ${\mathcal S}$):
\begin{itemize}
\item[1.] If $Y_n = (y,t)$ with $y \in S$ 
and $0 \le t < T_{corr}-1$, 
pick $y'$ from ${\mathbb P}_y(X_1 \in \cdot)$. 
If $y'\in S$, 
let $Y_{n+1} = (y',t+1)$; otherwise 
let $Y_{n+1} = (y',0)$.

\item[2.] If $Y_n = (y,T_{corr}-1)$ with $y \in S$, 
pick $z$ from the QSD $\nu$ in $S$, and pick $y'$ 
from ${\mathbb P}_z(X_1 \in \cdot)$. If 
$y' \in S$, let $Y_{n+1} = (y',T_{corr})$; 
otherwise let $Y_{n+1} = (y',0)$.

\item[3.] If $Y_n = (y,T_{corr})$ with $y \in S$, 
pick $y'$ from ${\mathbb P}_y(X_1 \in \cdot)$. 
If $y' \in S$, let 
$Y_{n+1} = (y',T_{corr})$; otherwise 
let $Y_{n+1} = (y',0)$.
\end{itemize}
The process $(Y_n)_{n\ge 0}$ is Markovian on state space
$(\Omega_Y, {\mathcal F}_Y)$, where 
\begin{equation*}
\Omega_Y =\Omega \times {\mathbb Z},\quad {\mathcal F}_Y = {\mathcal F} \otimes 2^{\mathbb Z}.
\end{equation*}
The following result is immediate from construction.
\begin{lemma}\label{L3a}
If $\pi_1(Y_0) \sim {\tilde X}_0$ and $\pi_2(Y_0) = 0$, then 
$(\pi_1(Y_n))_{n\ge 0}\sim({\tilde X}_n)_{n\ge 0}$.
\end{lemma}
Note that for the processes to have the same law, the 
counter of the extended process must start at zero. 
Lemmas~\ref{L2}-~\ref{L3a} give the
following relationship between the extended 
process and the original Markov chain: 
\begin{lemma}\label{L3}
If $\pi_1(Y_0) \sim X_0$ and $\pi_2(Y_0) = 0$, then for 
every $n \ge 0$, $\pi_1(Y_n) \sim X_n$.
\end{lemma}

\subsubsection{Harris Chains}\label{Sec:Harris}

Let $(Z_n)_{n\ge 0}$ be a Markov chain 
on a standard Borel state space $(\Sigma,{\mathcal E})$. 
The process 
$(Z_n)_{n\ge 0}$ is a {\em Harris chain} 
if there exists $\epsilon > 0$, 
$A, B \in {\mathcal E}$, 
and a probability measure $\rho$ on $(\Sigma, {\mathcal E})$ 
supported in $B$ such that 
\begin{itemize}
\item[(i)] For all $x \in \Sigma$, we have ${\mathbb P}_x(\,\inf\{n\ge 0\,:\, Z_n \in A\} < \infty)>0$;
\item[(ii)] For all $x \in A$ and $C \in {\mathcal E}$ 
with $C \subset B$, we have ${\mathbb P}_x(Z_1 \in C) \ge \epsilon \rho(C)$.
\end{itemize}
See for instance Chapter 5 of~\cite{Durrett}. 
Intuitively, starting at any point 
in $A$, with probability at least $\epsilon$, the process 
is distributed according to $\rho$ 
after one time step. This allows 
ergodicity of the chain to be studied using 
ideas similar to the case where the state 
space is discrete. 
The trick is to consider an auxiliary process 
$({\bar Z}_n)_{n\ge 0}$ with values in
${\bar \Sigma} := \Sigma \cup \{\sigma\}$, 
where $\sigma$ corresponds to being distributed 
according to $\rho$ on $B$. More precisely:
\begin{itemize}
\item[1.] If ${\bar Z}_n = x$ and $x \in \Sigma \setminus A$, 
pick $y$ from 
${\mathbb P}_x(Z_1 \in \cdot)$ and let ${\bar Z}_{n+1} = y$.

\item[2.] If ${\bar Z}_n = x$ and $x \in A$: 
with probability $\epsilon$, let ${\bar Z}_{n+1} = \sigma$; 
with probability $1-\epsilon$, pick $y$ from 
$(1-\epsilon)^{-1}({\mathbb P}_x(Z_1 \in \cdot)-\epsilon \rho(\cdot))$ 
and let ${\bar Z}_{n+1} = y$.

\item[3.] If ${\bar Z}_n = \sigma$, pick $x$ from $\rho(\cdot)$. 
Then pick ${\bar Z}_{n+1}$ as in 1-2.
\end{itemize}
So $({\bar Z}_n)_{n\ge 0}$ is Markov on 
$({\bar \Sigma}, {\bar {\mathcal E}})$, where 
$\bar {\mathcal E}$ consists of sets of the form 
$C$ and $C \cup \{\sigma\}$ for $C \in {\mathcal E}$. 
The following result (see~\cite{Durrett}) relates 
the auxiliary process to the original process.
\begin{lemma}\label{L4}
Let $f:{\Sigma} \to {\mathbb R}$ be bounded and 
measurable, and 
define ${\bar f}: {\bar \Sigma} \to {\mathbb R}$ by 
\begin{equation*}
{\bar f}(x) = \begin{cases} f(x),& x \in \Sigma\\
\int_B f\,d\rho,& x = \sigma \end{cases}
\end{equation*}
Then for any
probability measure $\xi$ on 
$({\Sigma},{{\mathcal E}})$ and any $n \ge 0$, 
\begin{equation*}
{\mathbb E}_\xi[f(Z_n)] = {\mathbb E}_{\bar \xi}[{\bar f}({\bar Z}_n)], 
\end{equation*}
where ${\bar \xi}$ is the probability measure on 
$({\bar \Sigma},{\bar {\mathcal E}})$ defined by 
${\bar \xi}(A) = \xi(A)$ for $A \in {\mathcal E}$,
and ${\bar \xi}(\sigma)=0$.
\end{lemma}
The following theorem gives sufficient conditions for 
the Harris chain to be ergodic. Note that the 
conditions are in terms of the auxiliary chain.
\begin{lemma}\label{L5}
Let $(Z_n)_{n\ge 0}$ be a Harris chain on $(\Sigma,{\mathcal E})$ 
with auxiliary chain $({\bar Z}_n)_{n\ge 0}$. Assume that
\begin{equation*}
\sum_{n =1}^\infty {\mathbb P}_\sigma({\bar Z}_n = \sigma) = \infty
\end{equation*}
and
\begin{equation*}
\textup{g.c.d.}\{n\ge 0\,:\, {\mathbb P}_\sigma({\bar Z}_n = \sigma) > 0\} = 1.
\end{equation*}
Then there exists a (unique) measure $\eta$ 
on $(\Sigma, {\mathcal E})$ such that 
for any probability measure $\xi$ 
on $(\Sigma, {\mathcal E})$ and any 
bounded measurable function 
$f:\Sigma \to {\mathbb R}$, 
\begin{equation*}
{\mathbb P}_\xi\left(\lim_{n\to \infty}\frac{f(Z_0) + \ldots + f(Z_{n-1})}{n} = \int_\Sigma f \,d\eta\right) = 1.
\end{equation*}
Moreover, for any probability measure $\xi$ 
on $(\Sigma, {\mathcal E})$,
\begin{equation*}
\lim_{n \to \infty} ||{\mathbb P}_\xi(Z_n \in \cdot) - \eta|| = 0.
\end{equation*}
\end{lemma}
Proof of Lemma~\ref{L5} can be found in Chapter~5 of~\cite{Durrett} and 
Chapter~7 of~\cite{Douc}.

\subsubsection{Ergodicity of the extended process}\label{sec:EE}

In Theorem~\ref{T2} below, ergodicity of the extended process 
$(Y_n)_{n\ge 0}$ is proved. Before proceeding, three preliminary results, 
Lemmas~\ref{L6aa}-~\ref{L6} below, are required. Define 
\begin{equation}\label{tau}
\tau = \inf\{n\ge 0\,:\, \pi_2(Y_n) = 0\}.
\end{equation}
From Lemma~\ref{L3}, $\tau$ can be thought of as the first 
time $n$ at which the law of $\pi(Y_n)$ synchronizes 
with that of $X_n$. 
\begin{lemma}\label{L6aa}
For any probability measure $\xi$ on 
$(\Omega_Y, {\mathcal F}_Y)$ and any $t \ge T_{corr}+1$, 
\begin{equation*}
{\mathbb P}_\xi(\tau \le t) \ge \delta.
\end{equation*}
\end{lemma}
\begin{proof}
Let $t \ge T_{corr}+1$ and define
\begin{equation*}
\kappa = \inf\{n\ge 0\,:\, \pi_2(Y_n) = T_{corr}\}.
\end{equation*}
Note that if $\tau> \kappa$, then $\kappa \le T_{corr}$ and so $\kappa + 1 \le t$. 
On the other hand, if $\tau \le \kappa$, then $\tau \le T_{corr}$ and 
so $\tau \le t-1 < t$. Thus, 
\begin{align}\begin{split}\label{ab1}
{\mathbb P}_\xi(\tau \le t) &= 
{\mathbb P}_\xi(\tau \le t\,|\, \tau \le \kappa){\mathbb P}_\xi(\tau \le \kappa) + {\mathbb P}_\xi(\tau \le t\,|\, \tau > \kappa){\mathbb P}_\xi(\tau > \kappa)\\
&={\mathbb P}_\xi(\tau \le \kappa)+ {\mathbb P}_\xi(\tau \le t\,|\, \tau > \kappa){\mathbb P}_\xi(\tau > \kappa)\\
&\ge {\mathbb P}_\xi(\tau \le \kappa)+ {\mathbb P}_\xi(\tau = \kappa + 1\,|\, \tau > \kappa){\mathbb P}_\xi(\tau > \kappa).
\end{split}
\end{align}
By Assumption~\ref{A3}, for any $(x,t) \in \Omega_Y$, 
\begin{align}\begin{split}\label{ab2}
 {\mathbb P}_{(x,t)}(\tau = \kappa+1\,|\,\tau > \kappa) 
&= \sum_{k = 0}^{T_{corr}}{\mathbb P}_{(x,t)}(\tau = k+1\,|\,\kappa = k,\,\tau > k)
{\mathbb P}(\kappa = k\,|\, \tau > \kappa)\\
&= \sum_{k = 0}^{T_{corr}}{\mathbb P}_\nu(X_1 \notin S){\mathbb P}(\kappa = k\,|\, \tau > \kappa)\\
&= {\mathbb P}_\nu(X_1 \notin S) \ge \delta,
\end{split}
\end{align}
where $\nu$ is the QSD in $S$, with $S \ni x$. Combining~\eqref{ab1} 
and~\eqref{ab2} and using the fact that $\delta \in (0,1]$, 
\begin{align*}
{\mathbb P}_\xi(\tau \le t) &\ge {\mathbb P}_\xi(\tau \le \kappa) 
+ \delta {\mathbb P}_\xi(\tau > \kappa) \\
&= \delta + (1-\delta){\mathbb P}_\xi(\tau \le \kappa) \\
&\ge \delta.
\end{align*}

\end{proof}

For the remainder of 
Section~\ref{sec:EE}, fix $S \in {\mathcal S}$
satisfying Assumption~\ref{A3}, and define 
\begin{equation}\label{S_Y}
S_Y = S \times \{T_{corr}\}.
\end{equation}
\begin{lemma}\label{L6a}
Let $\xi$ be any probability measure on 
$(\Omega_Y,{\mathcal F}_Y)$ with support 
in $S \times \{0,\ldots,T_{corr}\}$. Then for all $n \ge T_{corr}$,
\begin{equation*}
{\mathbb P}_\xi(Y_{n} \in S_Y) \ge \alpha^{n}.
\end{equation*}
\end{lemma}
\begin{proof}
By Assumption~\ref{A3}, ${\mathbb P}_x(X_1 \in S) \ge \alpha$ 
whenever $x \in S$. By definition of the extended
process $(Y_n)_{n\ge 0}$, the following holds. 
First, for any $x \in S$ 
and any $t \in \{0,\ldots,T_{corr}-2\}$,  
\begin{equation}\label{es1}
{\mathbb P}_{(x,t)}(Y_1 \in S \times \{t+1\}) = {\mathbb P}_x(X_1 \in S) 
\ge \alpha,
\end{equation}
Second, for any $x \in S$, 
\begin{equation}\label{es2}
{\mathbb P}_{(x,T_{corr}-1)}(Y_1 \in S_Y)
= \int_S {\mathbb P}_y(X_1 \in S)\,\nu(dy) \ge \int_S \alpha\,\nu(dy) = \alpha 
\end{equation}
Third, for any $x \in S$, 
\begin{equation}\label{es3}
{\mathbb P}_{(x,T_{corr})}(Y_1 \in S_Y) = {\mathbb P}_x(X_1 \in S) \ge \alpha.
\end{equation}
Let $n \ge T_{corr}$. For 
any $x \in S$ and $t \in \{0,\ldots,T_{corr}\}$, due 
to~\eqref{es1},~\eqref{es2} and~\eqref{es3},
\begin{align*}
\alpha^n &\le 
\prod_{j=1}^{T_{corr}-t} 
{\mathbb P}_{(x,t)}\left(Y_j \in S \times\{t+j\}\,|\,Y_{j-1} \in S \times\{t+j-1\}\right)\\
&\quad\times \prod_{j=T_{corr}-t+1}^n {\mathbb P}_{(x,t)}\left(Y_j \in S_Y\,|\, Y_{j-1} \in S_Y\right)
\\&= {\mathbb P}_{(x,t)}\left(\bigcap_{j=1}^{T_{corr}-t} \{Y_j \in S \times\{t+j\}\},\, 
\bigcap_{j = T_{corr}-t+1}^n \{Y_j \in S_Y\}\right)\\
&\le  {\mathbb P}_{(x,t)}(Y_n \in S_Y),
\end{align*}
where by convention the product and intersection from 
$j= T_{corr}-t+1$ to $j = n$ do not appear above 
if $n = T_{corr}$ and $t = 0$.
\end{proof}

Lemmas~\ref{L6aa}-~\ref{L6a} lead to the following.
\begin{lemma}\label{L6}
There exists $N \ge 0$ and $c>0$ 
such that for all 
probability measures $\xi$ on $(\Omega_Y, {\mathcal F}_Y)$ 
and all $n \ge N$,
\begin{equation*}
{\mathbb P}_{\xi}(Y_n \in S_Y) \ge c.
\end{equation*}
\end{lemma}
\begin{proof}
Fix a probability measure $\xi$ 
on $(\Omega_Y, {\mathcal F}_Y)$. Since $\mu(S)>0$, 
by Assumption~\ref{A2} one may choose $N'' \ge 0$ 
and $c'>0$ such that for all probability 
measures $\zeta$ on $(\Omega, {\mathcal F})$ 
and all $n \ge N''$, 
\begin{equation}\label{unifbound}
{\mathbb P}_{\zeta}(X_n \in S) \ge c'.
\end{equation}
Let $N' = N'' + T_{corr} + 1$ and define 
$\tau$ as in~\eqref{tau}. 
For $j \ge 0$, define 
probability measures $\xi_j$ on $(\Omega, {\mathcal F})$ 
by, for $A \in {\mathcal F}$, 
\begin{equation*}
\xi_j(A) = {\mathbb P}_\xi(\pi_1(Y_j) \in A\,|\, \tau = j).
\end{equation*}
By Lemma~\ref{L3} and~\eqref{unifbound}, for all $j \in \{0,\ldots,T_{corr}+1\}$ 
and $n \ge N'$,
\begin{align*}
{\mathbb P}_\xi(\pi_1(Y_n) \in S\,|\, \tau = j) 
&= \int_\Omega {\mathbb P}_\xi(\pi_1(Y_n) \in S\,|\, \pi_1(Y_j) = x,\,\tau = j) \,
\xi_j(dx)\\
&= \int_\Omega {\mathbb P}_{(x,0)}(\pi_1(Y_{n-j}) \in S)\,\xi_j(dx)\\
&= \int_\Omega {\mathbb P}_{x}(X_{n-j} \in S)\,\xi_j(dx)\\
&= {\mathbb P}_{\xi_j}(X_{n-j} \in S)
 \ge c'.
\end{align*}
So by Lemma~\ref{L6aa}, for all $n \ge N'$, 
\begin{align}\begin{split}\label{est1}
{\mathbb P}_\xi(\pi_1(Y_n) \in S) 
&\ge \sum_{j=0}^{T_{corr}+1} {\mathbb P}_\xi(\pi_1(Y_n) \in S\,|\,\tau = j)\,{\mathbb P}_\xi(\tau = j)\\
&\ge c'{\mathbb P}_\xi(\tau \le T_{corr}+1) \\
&\ge c'\delta.\end{split}
\end{align}
Let $N = N' + T_{corr}$ and fix $n \ge N$.
Define a probability 
measure $\phi_n$ on $(\Omega_Y,{\mathcal F}_Y)$ with 
support in $S \times\{0,\ldots,T_{corr}\}$ 
by, for $A \in {\mathcal F}$ and $t \in \{0,\ldots,T_{corr}\}$, 
\begin{equation*}
\phi_n(A,t) = {\mathbb P}_{\xi}(Y_{n-T_{corr}} \in A \times\{t\}\,|\,\pi_1(Y_{n-T_{corr}}) \in S)
\end{equation*}
By Lemma~\ref{L6a} and~\eqref{est1}, 
\begin{align*}
{\mathbb P}_\xi(Y_n \in S_Y) &\ge 
{\mathbb P}_\xi(Y_n \in S_Y\,|\, \pi_1(Y_{n-T_{corr}}) \in S)\,
{\mathbb P}_\xi(\pi_1(Y_{n-T_{corr}}) \in S)\\
&= {\mathbb P}_{\phi_n}(Y_{T_{corr}} \in S_Y)\,{\mathbb P}_\xi(\pi_1(Y_{n-T_{corr}}) \in S)\\
&\ge {\mathbb P}_{\phi_n}(Y_{T_{corr}} \in S_Y)\, c'\delta \\
&\ge \alpha^{T_{corr}} c'\delta.
\end{align*}
Taking $c = \alpha^{T_{corr}} c'\delta$ completes the proof.
\end{proof}

Finally ergodicity of 
the extended process $(Y_n)_{n\ge 0}$ 
can be proved, using the tools of Section~\ref{Sec:Harris}.
\begin{theorem}\label{T2}
There exists a (unique) measure $\mu_Y$ on 
$(\Omega_Y, {\mathcal F}_Y)$ 
such that for any probability measure $\xi$ on 
$(\Omega_Y, {\mathcal F}_Y)$ and any 
bounded measurable function 
$f:\Omega_Y \to {\mathbb R}$, 
\begin{equation*}
{\mathbb P}_\xi\left(\lim_{n\to \infty}\frac{f(Y_0) + \ldots + f(Y_{n-1})}{n} = \int_\Omega f \,d\mu_Y \right) = 1.
\end{equation*}
Moreover, for any probability measure $\xi$ on 
$(\Omega_Y, {\mathcal F}_Y)$,
\begin{equation*}
\lim_{n\to \infty} ||{\mathbb P}_\xi(Y_n \in \cdot) - \mu_Y|| = 0.
\end{equation*}
\end{theorem}
\begin{proof}
First, it is claimed $(Y_n)_{n\ge 0}$ is a Harris chain. 
Recall that $S$ and $S_Y$ are defined as in~\eqref{S_Y}.
Lemma~\ref{L6} shows that for any $(x,t) \in \Omega_Y$, 
\begin{equation}\label{cond1}
{\mathbb P}_{(x,t)}\left(\,\inf\{n\ge 0\,:\, Y_n \in S_Y\} < \infty\right) > 0.
\end{equation}
Define a probability measure $\rho$ on 
$(\Omega_Y, {\mathcal F}_Y)$ with support in $S_Y$ 
by: for $A \in {\mathcal F}$ and $t \in \{0,\ldots,T_{corr}\}$, 
\begin{equation*}
\rho(A,t) = \begin{cases} \lambda(A),& t = T_{corr}\\
0,& \hbox{else} \end{cases}
\end{equation*}
Let $C \in {\mathcal F}_Y$ with $C \subset S_Y$. 
Then $C = A \times\{T_{corr}\}$ with $A \in {\mathcal F}$, 
$A \subset S$. 
From Assumption~\ref{A3}, for any $(x,t) \in S_Y$, 
\begin{equation}\label{cond2}
{\mathbb P}_{(x,t)}(Y_1 \in C) = {\mathbb P}_x(X_1 \in A) 
\ge \alpha\lambda(A) = \alpha \rho(C).
\end{equation}
One can check $(Y_n)_{n\ge 0}$ is a 
Harris chain by taking $A = B = S_Y$, $\epsilon = \alpha$, 
and $\rho$ as above in the definition of Harris chains
in Section~\ref{Sec:Harris}.

Next it is proved that $(Y_n)_{n\ge 0}$ is ergodic. Let 
$({\bar Y}_n)_{n\ge 0}$ be the auxiliary chain 
defined as in Section~\ref{Sec:Harris}. Note that 
\begin{equation*}
{\mathbb P}_{\sigma}({\bar Y}_1 = \sigma) = \alpha.
\end{equation*}
This shows the second assumption of Lemma~\ref{L5} holds, 
that is,  
\begin{equation*}
\textup{g.c.d.}\{n\ge 0\,:\, {\mathbb P}_\sigma({\bar Y}_n = \sigma) > 0\} = 1,
\end{equation*}
since $1$ is in the set. 
Consider now the first assumption. It must be shown that
\begin{equation}\label{recurrent}
\sum_{n=1}^\infty {\mathbb P}_\sigma({\bar Y}_n = \sigma) = \infty.
\end{equation}
By Lemma~\ref{L6}, one can choose $N\ge 0$ and $c>0$ such 
for all probability measures $\xi$ on $(\Omega_Y,{\mathcal F}_Y)$ 
and all $n \ge N$, 
\begin{equation}\label{BD}
{\mathbb P}_\xi(Y_n \in S_Y) \ge c.
\end{equation}
Define a probability measure $\bar \xi$ on 
$({\bar \Omega}_Y, {\bar {\mathcal F}}_Y)$ by
\begin{equation*}
 {\bar \xi}(A) = {\mathbb P}_\sigma({\bar Y}_1 
\in A\,|\, {\bar Y}_1 \ne \sigma) \hbox{ for }A \in {\mathcal F}_Y,\quad 
{\bar \xi}(\sigma) = 0,
\end{equation*}
and let $\xi$ be the probability 
measure on $(\Omega_Y, {\mathcal F}_Y)$ which 
is the restriction of $\bar \xi$ to $\Omega_Y$. 
By~\eqref{BD} and Lemma~\ref{L4} with $f = 1_{S_Y}$, 
for all $n\ge N$, 
\begin{align}\begin{split}\label{last}
c &\le {\mathbb P}_\xi(Y_n \in S_Y)\\
&= {\mathbb P}_{\bar \xi}({\bar Y}_n \in S_Y) 
+ {\mathbb P}_{\bar \xi}({\bar Y}_n = \sigma)\\
&= {\mathbb P}_{\bar \xi}({\bar Y}_n \in S_Y \cup \{\sigma\}).\end{split}
\end{align}
Using~\eqref{last}, for $n \ge N+1$, 
\begin{align}\begin{split}\label{BD2}
{\mathbb P}_\sigma({\bar Y}_n \in S_Y \cup \{\sigma\}) 
&\ge {\mathbb P}_{\sigma}({\bar Y}_n \in S_Y \cup \{\sigma\}\,|\, {\bar Y}_1 \ne \sigma){\mathbb P}_\sigma(Y_1 \ne \sigma)
\\&= (1-\alpha){\mathbb P}_{\sigma}({\bar Y}_n \in S_Y \cup \{\sigma\}\,|\, {\bar Y}_1 \ne \sigma)\\
&= (1-\alpha){\mathbb P}_{\bar \xi}({\bar Y}_{n-1} \in S_Y \cup \{\sigma\})\\
&\ge (1-\alpha)c.\end{split}
\end{align}
Now by~\eqref{BD2}, for $n \ge N+2$, 
\begin{align*}
{\mathbb P}_\sigma({\bar Y}_n = \sigma) 
&\ge {\mathbb P}_\sigma({\bar Y}_n = \sigma \,|\, {\bar Y}_{n-1} \in S_Y \cup \{\sigma\})
{\mathbb P}_\sigma({\bar Y}_{n-1} \in S_Y \cup \{\sigma\})\\
&\ge {\mathbb P}_\sigma({\bar Y}_n = \sigma \,|\, {\bar Y}_{n-1} \in S_Y \cup \{\sigma\})
(1-\alpha)c\\
&= \alpha(1-\alpha)c > 0.
\end{align*}
Thus~\eqref{recurrent} holds. The result 
now follows from Lemma~\ref{L5}.
\end{proof}

\subsubsection{Ergodicity of the ParRep process with one replica}\label{sec:EP}
Next, ergodicity of $({\tilde X}_n)_{n\ge 0}$, the 
ParRep process with one replica, is proved.
\begin{theorem}\label{T3}
 For all probability measures $\xi$ on $(\Omega, {\mathcal F})$ 
and all bounded measurable functions 
$f:\Omega \to {\mathbb R}$, 
\begin{equation*}
{\mathbb P}_\xi\left(\lim_{n\to \infty}\frac{f({\tilde X}_0) + \ldots + f({
\tilde X}_{n-1})}{n} = \int_\Omega f \,d\mu\right) = 1.
\end{equation*}\end{theorem}
\begin{proof}
Fix a probability measure $\xi$ on $(\Omega,{\mathcal F})$ 
and a bounded measurable function $f:\Omega \to {\mathbb R}$. 
Define $f_Y:\Omega_Y\to {\mathbb R}$ by 
\begin{equation*}
f_Y = f \circ \pi_1,
\end{equation*}
and define a probability measure 
$\xi_Y$ on $(\Omega_Y,{\mathcal F}_Y)$ by, 
for $A \in {\mathcal F}$ and $t \in \{0,\ldots,T_{corr}\}$, 
\begin{equation}\label{xiy}
\xi_Y(A,t) = \begin{cases} \xi(A),& t = 0\\ 0, & t \in \{1,\ldots,T_{corr}\}\end{cases}
\end{equation}
By Theorem~\ref{T2}, there exists 
a (unique) measure $\mu_Y$ on $(\Omega_Y,{\mathcal F}_Y)$ 
such that 
\begin{equation}\label{erg1Y}
{\mathbb P}_{\xi_Y}\left(\lim_{n\to \infty}\frac{f_Y({Y}_0) + \ldots + f_Y({Y}_{n-1})}{n} = \int_{\Omega_Y} f_Y \,d\mu_Y\right) = 1
\end{equation}
and 
\begin{equation}\label{erg2Y}
\lim_{n\to \infty} ||{\mathbb P}_{\xi_Y}(Y_n \in \cdot) - \mu_Y|| = 0.
\end{equation}
Define a measure $\mu'$ on $(\Omega, {\mathcal F})$ by, 
for $A \in {\mathcal F}$,  
\begin{equation*}
\mu'(A) = \sum_{t=0}^{T_{corr}} \mu_Y(A,t).
\end{equation*}
From this and the definition of $f_Y$, 
\begin{equation*}
\int_{\Omega_Y} f_Y \,d\mu_Y = \int_\Omega f\,d\mu'.
\end{equation*}
So by Lemma~\ref{L3a} and~\eqref{erg1Y}, 
\begin{align}\begin{split}\label{ergtildeX}
1 &= {\mathbb P}_{\xi_Y}\left(\lim_{n\to \infty}\frac{f_Y({Y}_0) + \ldots + f_Y({Y}_{n-1})}{n} = \int_{\Omega_Y} f_Y \,d\mu_Y\right)\\ 
&= {\mathbb P}_\xi\left(\lim_{n\to \infty}\frac{f({\tilde X}_0) + \ldots + f({
\tilde X}_{n-1})}{n} = \int_\Omega f\,d\mu'\right).\end{split}
\end{align}
Also, by Lemma~\ref{L3} and~\eqref{erg2Y}, 
\begin{align*}
0 &= \lim_{n\to \infty}\, \sup_{A \in {\mathcal F}}|{\mathbb P}_{\xi_Y}(Y_n \in A \times \{0,\ldots, T_{corr}\}) - \mu_Y(A \times\{0,\ldots,T_{corr}\})| \\
&= \lim_{n\to \infty}\, \sup_{A \in {\mathcal F}}|{\mathbb P}_{\xi}({X}_n \in A) - \mu'(A)|\\
&= \lim_{n\to \infty} ||{\mathbb P}_{\xi}({X}_n \in \cdot) - \mu'||.
\end{align*}
Using Assumption~\ref{A2} one can conclude $\mu = \mu'$. 
So from~\eqref{ergtildeX}, 
\begin{equation*}
{\mathbb P}_\xi\left(\lim_{n\to \infty}\frac{f({\tilde X}_0) + \ldots + f({
\tilde X}_{n-1})}{n} = \int_\Omega f\,d\mu\right) = 1.
\end{equation*}
\end{proof}

\subsubsection{Proof of main result}\label{sec:MP}

Here the main result, 
Theorem~\ref{MainTheorem}, is finally proved. The idea is 
to use ergodicity of 
$({\tilde X}_n)_{n\ge 0}$ along with the 
fact that the {\em average} value of the 
contribution to $f_{sim}$ from
a Parallel Step of Algorithm~\ref{alg1} 
does not depend on the number of replicas. 
A law of large numbers applied to the contributions 
to $f_{sim}$ from all the Parallel Steps will
then be enough to conclude. Note 
that the law of $(f_{sim})_{T_{sim} \ge 0}$ 
depends on the number $N$ of replicas, 
but this is not indicated explicitly.

\begin{proof}[Proof of Theorem~\ref{MainTheorem}.]
Fix a probability measure $\xi$ on $(\Omega,{\mathcal F})$ and 
a bounded measurable function $f:{\Omega} \to {\mathbb R}$. 
Define ${\hat f}:\Omega_Y \to {\mathbb R}$ by 
\begin{equation*}
{\hat f}(x,t) = \begin{cases} f(x),& t \in \{0,\ldots,T_{corr}-1\}\\
0,& t = T_{corr}\end{cases}
\end{equation*}
Let Algorithm~\ref{alg1} start at $\xi$.
The quantity $f_{sim}$ will be decomposed 
into contributions from the 
Decorrelation Step and the Parallel Step. 
Let $f_{sim}^{corr}$ denote the contribution to $f_{sim}$ 
from the Decorrelation Step up to time $T_{sim}$, and 
let $f_{sim}^{par}$ denote the contribution to $f_{sim}$ 
from the Parallel Step up to time $T_{sim}$. Thus, 
\begin{equation}\label{pt}
 (f_{sim})_{T_{sim} \ge 0} = \left(f_{sim}^{corr}+f_{sim}^{par}\right)_{T_{sim} \ge 0}.
\end{equation}
Let $(Y_n)_{n\ge 0}$ start at $Y_0 \sim \xi_Y$, with 
$\xi_Y$ defined as in~\eqref{xiy}. 
Because the starting points $x_1,\ldots,x_N$ sampled in 
the Dephasing Step are independent of the history of algorithm, 
each Parallel Step -- in particular the pair
$(\tau_{acc},X_{\tau^K}^K)$ -- is 
independent of the history of the algorithm. 
This and Theorem~\ref{T0} imply 
that $(f_{sim}^{corr})_{T_{corr}\ge 0}$ has 
the same law for every number of replicas 
$N$. In particular when $N=1$, 
from Lemma~\ref{L3a},
\begin{equation}\label{pt1}
 (f_{sim}^{corr})_{T_{sim}\ge 0} 
 \sim \left(\sum_{i=0}^{T_{corr}}{\hat f}(Y_i)\right)_{T_{sim} \ge 0}.
\end{equation}
Meanwhile, from the preceding independence argument, 
\begin{equation}\label{pt2}
 (f_{sim}^{par})_{T_{sim} \ge 0} 
\sim \left(\sum_{S \in {\mathcal S}}\sum_{i=1}^{n_{S,T_{sim}}} \theta_{S,N}^{(i)} \right)_{T_{sim} \ge 0},
\end{equation}
where $\{\theta_{S,N}^{(i)}\}_{i=1,2,\ldots}$ are iid random variables and 
$n_{S,T_{sim}}$ counts the number of sojourns of 
$(Y_n)_{n\ge 0}$ in $S \times\{T_{corr}\}$ 
by time $T_{sim}$:
\begin{equation*}
n_{S,T_{sim}} = \#\{1\le n < T_{sim}\,:\, Y_n \in S\times \{T_{corr}\},\,Y_{n-1} \in S \times \{T_{corr}-1\}\}. 
\end{equation*}
From Idealization~\ref{A0} and 
Definition~\ref{D1}, each term in the sum in~\eqref{sum1} or~\eqref{sum2} 
of the Parallel Step has expected value $\int_S f\,d\nu$. So from  
linearity of expectation and Theorems~\ref{T0}, for any 
number $N$ of replicas, 
\begin{align}\begin{split}\label{theta}
{\mathbb E}[\theta_{S,N}^{(i)}] &= \left({\mathbb E}[\tau_{acc}]-1\right)\int_S f\,d\nu
\\&=\left({\mathbb P}_\nu(X_1 \notin S)^{-1}-1\right)\int_S f\,d\nu.\end{split}
\end{align}
Combining~\eqref{pt},~\eqref{pt1} and~\eqref{pt2}, for 
any number $N$ of replicas, 
\begin{align}\begin{split}\label{decomp}
\left(\frac{f_{sim}}{T_{sim}}\right)_{T_{sim}\ge 0} &\sim 
\left(\frac{1}{T_{sim}}\sum_{i=0}^{T_{sim}} {\hat f}(Y_i) 
+ \frac{1}{T_{sim}}\sum_{S \in {\mathcal S}}\sum_{j=0}^{n_{S,T_{sim}}}
 \theta_{S,N}^{(i)}\right)_{T_{sim} \ge 0},\end{split}
\end{align}
where it is assumed 
the processes on the left and right hand side 
of~\eqref{decomp} are independent.
Let $({\tilde X}_n)_{n\ge 0}$ start at 
${\tilde X}_0 \sim \xi$. From definition 
of $({\tilde X}_n)_{n\ge 0}$ and~\eqref{decomp}, 
when the number of replicas is $N=1$, 
\begin{align}\begin{split}\label{conv1}
\left(\frac{f_{sim}}{T_{sim}}\right)_{T_{sim}\ge 0} &\sim 
\left(\frac{f({\tilde X}_0) + \ldots + f({\tilde X}_{T_{sim}})}{T_{sim}} \right)_{T_{sim} \ge 0}\\
&\sim \left(\frac{1}{T_{sim}}\sum_{i=0}^{T_{sim}} {\hat f}(Y_i) 
+ \frac{1}{T_{sim}}\sum_{S \in {\mathcal S}}\sum_{j=0}^{n_{S,T_{sim}}} \theta_{S,1}^{(i)}\right)_{T_{sim} \ge 0},\end{split}
\end{align}
where the processes in~\eqref{conv1} are assumed independent.
Since $(Y_n)_{n\ge 0}$ is Markov, the number of 
time steps $n$ for which $Y_n \in S \times \{T_{corr}\}$ 
is either finite almost 
surely, or infinite almost surely. By Theorem~\ref{T0} 
and Assumption~\ref{A4}, the expected value 
of each of the sojourn times of $(Y_n)_{n\ge 0}$ in 
$S \times \{T_{corr}\}$ is $1/{\mathbb P}_\nu(X_1 \notin S)\le 1/\delta < \infty$, so the 
sojourn times are finite almost surely. This means that 
either $(Y_n)_{n\ge 0}$ has infinitely many sojourns in 
$S \times \{T_{corr}\}$ almost surely, 
or $(Y_n)_{n\ge 0}$ has finitely many sojourns in $S \times \{T_{corr}\}$ 
almost surely. Thus:
\begin{equation}\label{fininf}
\forall\,S \in {\mathcal S},  \hbox{ either }
{\mathbb P}_{\xi_Y}\left(\lim_{T_{sim} \to \infty} n_{S,T_{sim}} = \infty\right) = 1
 \hbox{ or }
{\mathbb P}_{\xi_Y}\left(\lim_{T_{sim} \to \infty} \frac{n_{S,T_{sim}}}{T_{sim}} = 0\right) =1.
\end{equation}
Define $\tau_S^{(0)} = 0$ and for $i = 1,2,\ldots$, 
\begin{align*}
&\tau_S^{(i)} = \inf\{n > \tau_S^{(i-1)}\,:\,Y_n \in S\times \{T_{corr}\},\,Y_{n-1} \in S \times \{T_{corr}-1\}\}\\
&\sigma_S^{(i)} = \tau_S^{(i)}-\tau_S^{(i-1)}.
\end{align*}
Note that $\{\sigma_S^{(i)}\}_{i=1,2,\ldots}$ are iid and 
\begin{equation*}
\frac{1}{n_{S,T_{sim}}}\sum_{i=1}^{n_{S,T_{sim}}} \sigma_S^{(i)} \le \frac{T_{sim}}{n_{S,T_{sim}}} \le \frac{1}{n_{S,T_{sim}}}\sum_{i=1}^{n_{S,T_{sim}}+1}\sigma_S^{(i)}.
\end{equation*}
If $n_{S,T_{sim}} \to \infty$ almost surely as $T_{sim} \to \infty$, then by the 
strong law of large numbers there is a constant $c'$ (depending 
on $S$) such that
\begin{equation}\label{SLLN}
{\mathbb P}_\xi\left(\lim_{T_{sim}\to \infty}\frac{T_{sim}}{n_{S,T_{sim}}} = c'\right) = 1. 
\end{equation}
From~\eqref{fininf},~\eqref{SLLN} and the strong law 
of large numbers, there is a constant $c$
such that  
\begin{equation}\label{conv2}
{\mathbb P}_{\xi_Y}\left(\lim_{T_{sim} \to \infty}\frac{1}{T_{sim}}\sum_{S \in {\mathcal S}}\sum_{j=0}^{n_{S,T_{sim}}} \theta_{S,N}^{(i)} = c\right) = 1,
\end{equation}
and due to~\eqref{theta} this $c$ {does not depend on the 
number of replicas $N$}. By using Theorem~\ref{T3} 
along with~\eqref{conv1} and \eqref{conv2}, 
\begin{equation}\label{conv3}
{\mathbb P}_{\xi_Y}\left(\lim_{T_{sim} \to \infty}\frac{1}{T_{sim}}\sum_{i=0}^{T_{sim}} {\hat f}(Y_i) = \int_\Omega f\,d\mu - c\right) = 1.
\end{equation}
Now using~\eqref{decomp},~\eqref{conv2} and~\eqref{conv3}, 
for any number $N$ of replicas, 
\begin{equation*}
{\mathbb P}_{\xi_Y}\left(\lim_{T_{sim} \to \infty}\frac{f_{sim}}{T_{sim}} = \int_\Omega f\,d\mu\right) = 1.
\end{equation*}
\end{proof}

\section*{Acknowledgement} 

The author would like to acknowledge Gideon Simpson 
(Drexel University), Tony Leli\`evre (Ecole des Ponts 
ParisTech) and Lawrence Gray (University of Minnesota) 
for fruitful discussions.

\end{document}